\documentclass[11pt]{article}
\usepackage{amsfonts}
\usepackage{amssymb}
\usepackage{amsthm}
\usepackage{amsmath}

\usepackage{bm}
\usepackage{graphicx}
\usepackage{amscd}
\usepackage{fullpage}
\usepackage{hyperref}
\usepackage{xypic}

\usepackage{MnSymbol}

\input xy
\xyoption{all}
\newtheorem{thm}{Proposition}[section]

\newtheorem{Thm}[thm]{Theorem}

\newtheorem{lem}[thm]{Lemma}
\newtheorem{defn}[thm]{Definition}

\title{On integral schemes over symmetric monoidal categories}
\author{Abhishek Banerjee}
\date{}

\numberwithin{equation}{section} \DeclareMathSizes{2}{10}{12}{13}
\begin{document}

\maketitle

\centerline{Dept. of Math, Indian Institute of Science, Bangalore - 560012, India.} 
\centerline{Email: abhishekbanerjee1313@gmail.com}

\begin{abstract}
We propose notions of ``Noetherian'' and ``integral'' for schemes  over an abelian symmetric monoidal category $(\mathcal C,\otimes,1)$. For Noetherian integral schemes, we construct a ``function field'' that is a commutative monoid object
of $(\mathcal C,\otimes,1)$. Our main result is a bijection between dominant rational maps
and morphisms of these ``function field objects''.

\end{abstract}

\section{Introduction}
Let $(\mathcal C,\otimes,1)$ be an abelian, closed symmetric monoidal category satisfying certain conditions. Then, a monoid object 
in $(\mathcal C,\otimes,1)$ is a triple $(A,m_A,e_A)$ consisting of a ``multiplication map''  $m_A:A\otimes A\longrightarrow A$ 
and  a ``unit map'' $e_A:1\longrightarrow A$  satisfying compatibility conditions analogous to an ordinary ring (see, for instance, 
\cite{BJT}). Accordingly, one may study the category $A-Mod$ of (left) $A$-module objects in $\mathcal C$ for such a monoid object $A$. For instance, the general Morita theory for monoids over symmetric monoidal categories has been studied by Vitale \cite{Vitale}.  Further, monoid objects in abelian model categories have been studied by Hovey \cite{Hovey}.  In \cite{AB4}, we have developed the theory of centers, centralizers as well as an analogue of usual localization in
commutative algebra  for monoids over $(\mathcal C,\otimes,1)$. In this paper, 
we continue our program of studying commutative algebra and algebraic geometry over symmetric monoidal categories
from \cite{AB4}, \cite{AB4.5}, \cite{AB5} and \cite{AB6}.  Our purpose in this note is to develop a good theory for integral schemes over symmetric monoidal categories. We mention here that our notion of a ``scheme'' over a symmetric monoidal category is
that given by To\"{e}n and Vaqui\'{e} \cite{TV}. The idea of doing algebraic geometry over a symmetric monoidal category $(\mathcal C,\otimes,1)$ has been developed by several authors (see, for instance, Deligne \cite{Del}, Hakim \cite{Hak}, To\"{e}n and Vaqui\'{e} \cite{TV}). When $\mathcal C=k-Mod$, the category of modules over a commutative ring $k$, we recover the usual
algebraic geometry of schemes over $Spec(k)$. 

\medskip 
More precisely, let $(\mathcal C,\otimes,1)$ be an abelian closed symmetric monoidal category that is also ``locally finitely generated''. The theory of locally finitely generated abelian categories and indeed, the theory of locally finitely generated
Grothendieck categories is very well developed in the literature and we refer the reader to \cite{Gar} for an introduction.  
We denote by $Comm(\mathcal C)$ the category of commutative monoid objects in $\mathcal C$.
 We will say that a commutative monoid object $A$ of $(\mathcal C,\otimes,1)$
is ``integral'' if $Hom_{A-Mod}(A,A)$ is an ordinary integral domain. However, this definition of integrality is really ``at the level of global sections'' which makes it difficult to extend results on usual integral schemes to schemes over
$(\mathcal C,\otimes,1)$. In this note, we realized that when this notion of integrality is strengthened
with a Noetherian assumption (see Definition \ref{D2}), we can obtain analogues of several important properties of integral schemes in usual algebraic 
geometry.  We mention here that we have explored the notion of ``Noetherian'' for monoids over symmetric monoidal categories
previously in \cite{AB5} and \cite{AB6}. However, our definition of  ``Noetherian''  in this note differs substantially from those presented  in \cite{AB5} and \cite{AB6}. Further,  our methods in this paper are a combination
of the methods used previously in \cite{AB5} and \cite{AB6}.  Our purpose is the following: for a Noetherian integral scheme $X$ over $(\mathcal C,\otimes,1)$ we construct
a commutative monoid object $K(X)$ in $(\mathcal C,\otimes,1)$ that plays the role of the ``function field'' of $X$.  In fact, we show that the category $K(X)-Mod$ of $K(X)$-module objects in $(\mathcal C,\otimes,1)$ satisfies several properties similar to those
of the category of vector spaces over a field. Thereafter, we make an additional assumption that  any commutative monoid
object $A$ in $\mathcal C$ is also a compact object of the category $A-Mod$. This is true, for instance, if $\mathcal C$ is taken
to be the category of sheaves of $\mathcal A$-modules, where $\mathcal A$ is a sheaf of commutative rings on a compact
topological space with a basis of compact open sets. This extra condition allows us to show that  a Noetherian scheme over $(\mathcal C,\otimes,1)$ is integral if and only if it is reduced and irreducible. Finally, our main result is the following:

\medskip
\begin{Thm}\label{Thm1.1} Let $(\mathcal C,\otimes,1)$ be an abelian, closed symmetric monoidal category that is also locally
finitely generated. Suppose that for any commutative monoid object $A$ in $\mathcal C$, $A$ is a compact
object of $A-Mod$.   Let $X$, $Y$ be  Noetherian integral schemes of finite type over $(\mathcal C,\otimes,1)$. Then, there exists a bijection between morphisms $K(X)\longrightarrow K(Y)$ in $Comm(\mathcal C)$ and dominant rational maps from $Y$ to $X$. 
\end{Thm}

\medskip
The background of our problem is as follows: in \cite{AB5} we have constructed, corresponding to every integral scheme 
$X$ over $(\mathcal C,\otimes,1)$ an ordinary field $k(X)$ without the Noetherian assumption. The elements of this field
are equivalence classes of pairs $(U,t_U)$ where $U=Spec(A)$ is a non-trivial open affine of $X$ and $t_U\in Hom_{A-Mod}(A,A)$ 
(see Section 2 for precise details). However, the association $X\mapsto k(X)$ loses a lot of information, i.e., for integral schemes
$X$, $Y$ over $(\mathcal C,\otimes,1)$, a morphism $k(X)\longrightarrow k(Y)$ of fields cannot be used to construct a rational map
$Y\dashedrightarrow X$ of schemes over $(\mathcal C,\otimes,1)$. In this paper we have obtained something stronger; a commutative
monoid object $K(X)$ of $\mathcal C$ for each Noetherian integral scheme $X$. As mentioned in Theorem \ref{Thm1.1} above,
morphisms $K(X)\longrightarrow K(Y)$ in $Comm(\mathcal C)$ correspond to dominant rational maps from $Y$ to $X$. Further,
we will see in Proposition \ref{Pop} that the ordinary field $k(X)$ constructed in \cite[$\S$ 4]{AB5} may be recovered from the commutative monoid object $K(X)\in Comm(\mathcal C)$
simply as $k(X)\cong Hom_{K(X)-Mod}(K(X),K(X))$. We also show that the category $K(X)-Mod$ satisfies several properties
similar to that of vector spaces over a field; for instance, any finitely generated $K(X)$-module is isomorphic
to a finite direct sum of copies of $K(X)$ (see Proposition \ref{PZ}). 

\medskip
We hope that the notion of the ``internal function field object''  in this paper will be the first step towards the
systematic development of related concepts such as  Weil divisors and Cartier divisors for schemes over
$(\mathcal C,\otimes,1)$ and eventually a good Chow theory for schemes over $(\mathcal C,\otimes,1)$. Further, 
the formalism of schemes over $(\mathcal C,\otimes,1)$ is the starting point for obtaining analogous results 
in homotopical algebraic geometry over an abelian symmetric monoidal category. In particular, we know that the category of simplicial modules
over a simplicial commutative ring is connected to the derived algebraic geometry of Lurie \cite{Lurie}. For more on
homotopical algebraic geometry, we refer the reader to the work of To\"{e}n and Vezzosi \cite{TV1} \cite{TV2}. 

\section{Integral schemes over $(\mathcal C,\otimes,1)$}

Let $(\mathcal C,\otimes,1)$ be an abelian, closed symmetric monoidal category. Then, for any 
$A$ in the category $Comm(\mathcal C)$ of commutative monoid objects  of $\mathcal C$, the 
category $A-Mod$ of $A$-modules is abelian and closed symmetric monoidal (see Vitale \cite{Vitale}). We  assume that filtered colimits commute with finite limits
in $A-Mod$. Let $Aff_{\mathcal C}:=Comm(\mathcal C)^{op}$ be the category of affine
schemes over $\mathcal C$ and denote by $Spec(A)$ the affine scheme corresponding
to $A\in Comm(\mathcal C)$. Then, To\"{e}n and Vaqui\'{e} \cite{TV} have introduced a Zariski topology on $Aff_{\mathcal C}$ as well as the notion
of Zariski open immersions in the category $Sh(Aff_{\mathcal C})$ of sheaves of sets
on $Aff_{\mathcal C}$. 

\medskip
\begin{defn}(see \cite[D\'{e}finition 2.15]{TV}) \label{D1} Let $X$ be an object of $Sh(Aff_{\mathcal C})$. Then, $X$ is said to
be a scheme over $(\mathcal C,\otimes,1)$ if there exists an epimorphism $
p:\coprod_{i\in I}X_i\longrightarrow X
$ in $Sh(Aff_{\mathcal C})$  where each $X_i$ is an affine scheme and each
$X_i\longrightarrow X$ is a Zariski open immersion.
\end{defn}

\medskip
By definition,  $M\in A-Mod$ is finitely generated if the functor 
$Hom_{A-Mod}(M,\_\_)$ preserves filtered colimits of monomorphisms in $A-Mod$. An $A$-module $M$ will be called
finitely presented if it can be expressed as a colimit $M\cong colim(0\longleftarrow 
A^m\overset{q}{\longrightarrow} A^n)$ for some morphism $q:A^m\longrightarrow A^n$ of free $A$-modules.  We now
assume that $\mathcal C$ is ``locally finitely generated'', i.e., any  
$M\in A-Mod$ may be expressed as a filtered colimit of its finitely generated submodules. 

\medskip
\begin{defn}\label{D2} A commutative monoid object $A\in Comm(\mathcal C)$ will be said
to be integral if $\mathcal E(A):=Hom_{A-Mod}(A,A)$ is an ordinary integral domain. Further,
$A\in Comm(\mathcal C)$ will be said to be Noetherian if $M\in A-Mod$ is  finitely generated if and only if $M$ is also finitely presented.

\medskip
A scheme $X$ over $(\mathcal C,\otimes,1)$ will be called integral (resp. Noetherian) if 
given any object $U=Spec(A)\longrightarrow X$ in the category $ZarAff(X)$ of Zariski open
affines of $X$, $A\in Comm(\mathcal C)$ is integral (resp. Noetherian). 
\end{defn} 

For integral $A\in Comm(\mathcal C)$ and any $0\ne s\in \mathcal E(A)$, we consider the localization $A_s:=colim(A\overset{s}{\longrightarrow}A\overset{s}{\longrightarrow}...)$
as in \cite[$\S$ 3]{AB4}. Then, we can consider the ``field of fractions'' $K(A)$ of 
$A$: 
\begin{equation}\label{E1}
K(A):=\underset{s\in \mathcal E(A)\backslash\{0\}}{colim}\textrm{ }A_s
\end{equation} having the universal property that any morphism $g:A\longrightarrow B$ in 
$Comm(\mathcal C)$ such that $\mathcal E(g)(s)$ is a unit in $\mathcal E(B)$ for each 
$0\ne s\in \mathcal E(A)$ induces a unique morphism from $K(A)$ to $B$ (see \cite[$\S$ 3]{AB4}). 

\begin{lem}\label{L3} If $A\in Comm(\mathcal C)$ is Noetherian and integral, every
$0\ne s\in \mathcal E(A)=Hom_{A-Mod}(A,A)$ is a monomorphism in $A-Mod$. 
\end{lem}

\begin{proof} We choose $0\ne s:A\longrightarrow A$ and let $i:I:=Ker(s)\longrightarrow A$ be the monomorphism of the kernel of $s$ into $A$. For any $g\in Hom_{A-Mod}(A,I)$, we see that
$s\circ (i\circ g)=0$. Since $\mathcal E(A)$ is an integral domain, we must have $g=0$. Therefore, 
$Hom_{A-Mod}(A,I)=0$ and hence $Hom_{A-Mod}(M,I)=0$ for any finitely presented $A$-module $M$.  Finally since any $M\in A-Mod$ can be expressed as a colimit of finitely presented
$A$-modules (since $A$ is Noetherian), we see that $Hom_{A-Mod}(M,I)=0$ for any $M\in A-Mod$. Hence, 
$I=0$. 
\end{proof}

\begin{lem}\label{Lem2} Let $A\in Comm(\mathcal C)$ be Noetherian and integral and let
$K(A)$ be as defined in \eqref{E1}. Then, $\mathcal E(K(A))=Hom_{K(A)-Mod}(K(A),K(A))$
is a field. 
\end{lem}

\begin{proof} It is clear that $A\cong colim(0\longleftarrow 0\longrightarrow A)$ is finitely presented in $A-Mod$. Since
$A$ is Noetherian, it follows that   $A$ is also finitely generated in $A-Mod$.  By definition, $A_s=colim(A\overset{s}{\longrightarrow}A\overset{s}{\longrightarrow}...)$ for each $0\ne s\in \mathcal E(A)$. Then, since each $0\ne s\in \mathcal E(A)$ is a monomorphism, it follows
that $\mathcal E(A_s)=Hom_{A_s-Mod}(A_s,A_s)\cong Hom_{A-Mod}(A,A_s)\cong \mathcal E(A)_s$. For any $0\ne t
\in \mathcal E(A)$, the monomorphism $t:A\longrightarrow A$ induces a monomorphism of filtered colimits $t:A_s
\longrightarrow A_s$. It follows that we have monomorphisms $A_s\longrightarrow A_{st}$ for $0\ne s,t\in \mathcal E(A)$. Again, 
considering the filtered colimit of monomorphisms defining $K(A)$ in \eqref{E1}, we get
$\mathcal E(K(A))=Hom_{K(A)-Mod}(K(A),K(A))\cong Hom_{A-Mod}(A,K(A))=Q(\mathcal E(A))$ where $Q(\mathcal E(A))$ is the
field of fractions of the integral domain $\mathcal E(A)$. 

\end{proof} 

\begin{thm}\label{P1} If $A\in Comm(\mathcal C)$ is Noetherian and integral, $K(A)$ is Noetherian. Further, $K(A)$ has no non-zero proper subobjects in $K(A)-Mod$.
\end{thm}

\begin{proof}  Since $A\cong colim(0\longleftarrow 0\longrightarrow A)$ is finitely presented and $A$ is Noetherian, $A$ is finitely generated in $A-Mod$. Then, the functor $Hom_{K(A)-Mod}(K(A),\_\_)
=Hom_{K(A)-Mod}(A\otimes_AK(A),\_\_)\cong Hom_{A-Mod}(A,\_\_)$ on the category $K(A)-Mod$
preserves filtered colimits of monomorphisms. It follows that $K(A)$ (and hence any finitely presented $K(A)$-module) is finitely generated in 
$K(A)-Mod$. 

\smallskip
Conversely, let $N$ be a finitely generated $K(A)$-module. We express $N$ as a filtered colimit $colim_{i\in I}N_i$ of its finitely presented $A$-submodules. The universal property of $K(A)$ implies that $A\longrightarrow K(A)$ is an epimorphism in $Comm(\mathcal C)$
 and
it follows that $K(A)\otimes_AK(A)\cong K(A)$. 
Then: 
\begin{equation}
N\cong N\otimes_{K(A)}K(A)\cong N\otimes_{K(A)}(K(A)\otimes_AK(A))
\cong N\otimes_AK(A)=colim_{i\in I}N_i\otimes_AK(A)
\end{equation} Since $K(A)$ is a flat $A$-module (see \cite[$\S$ 3]{AB4}), $\{N_i\otimes_AK(A)\}_{i\in I}$ is still a filtered system of monomorphisms. Since
$N$ is finitely generated in $K(A)-Mod$, it now follows that 
$N\cong N_{i_0}\otimes_AK(A)$ for some $i_0\in I$. Since $N_{i_0}$ is a finitely presented $A$-module, $N$ becomes a finitely presented $K(A)$-module. Thus, $K(A)$ is Noetherian. 

\smallskip
Finally, let $i:I\longrightarrow K(A)$ be a monomorphism in $K(A)-Mod$. Then, the morphism 
$i_{K(A)}:Hom_{K(A)-Mod}(K(A),I)\longrightarrow Hom_{K(A)-Mod}(K(A),K(A))$ is a monomorphism
of vector spaces over the field $\mathcal E(K(A))$. Hence, $i_{K(A)}$ is either $0$ or an 
isomorphism. If $i_{K(A)}=0$, then $i_M:Hom_{K(A)-Mod}(M,I)\longrightarrow Hom_{K(A)-Mod}(M,K(A))$
is $0$ for any finitely presented $M\in A-Mod$ and hence for any $M\in A-Mod$. Then, $i=0$ and hence $I=0$. Similarly, if $i_K$ is an isomorphism, it follows that so is $i$. 
\end{proof} 

\begin{thm}\label{P2.6} Let $A\in Comm(\mathcal C)$ be a Noetherian, integral commutative monoid object. Then, 
$K(A)$ is projective as a $K(A)$-module.
\end{thm} 

\begin{proof} We consider an epimorphism $e:M\longrightarrow N$ in $K(A)-Mod$ and any  morphism $0\ne f:K(A)\longrightarrow N$. 
We set $Q:=Im(f)$ and consider the following pullback in $K(A)-Mod$:
\begin{equation}
P:=lim(M\overset{e}{\longrightarrow}N\longleftarrow Q=Im(f))
\end{equation}
Since $K(A)-Mod$ is abelian, the pullback $e':P\longrightarrow Q$ of $e$  is an epimorphism. Further, since
$K(A)$ has no non-trivial subobjects in $K(A)-Mod$ and $f\ne 0$, we must have $Ker(f)=0$ and hence $Q=Im(f)\cong K(A)$. Since 
the induced morphism $e'_{K(A)}:Hom_{K(A)-Mod}(K(A),P)\longrightarrow Hom_{K(A)-Mod}(K(A),Q)\cong Hom_{K(A)-Mod}(K(A),K(A))
=\mathcal E(K(A))$ is a morphism of vector spaces over the field $\mathcal E(K(A))$, $e'_{K(A)}$ is either $0$ or an epimorphism. 
If  the morphism $e'_{K(A)}:Hom_{K(A)-Mod}(K(A),P)\longrightarrow Hom_{K(A)-Mod}(K(A),Q)$ is $0$, we can show as in the proof of 
Proposition \ref{P1} that the epimorphism $e':P\longrightarrow Q\cong K(A)$ is $0$. This contradicts the fact that $f\ne 0$.  Hence,
$e'_{K(A)}:Hom_{K(A)-Mod}(K(A),P)\longrightarrow Hom_{K(A)-Mod}(K(A),Q)$ must be an epimorphism. Thus, 
$f:K(A)\longrightarrow Q=Im(f)$ lifts to $P$ and it is clear that $f:K(A)\longrightarrow N$ lifts to $M$. 

\end{proof} 

\begin{lem}\label{LemmaX} Let $A\in Comm(\mathcal C)$ be a Noetherian, integral commutative monoid object.  Then, every monomorphism in $K(A)-Mod$ splits. 
\end{lem}

\begin{proof} We consider a monomorphism $i:M\longrightarrow N$ in $K(A)-Mod$ and the induced monomorphism $i_{K(A)}:
Hom_{K(A)-Mod}(K(A),M)\longrightarrow Hom_{K(A)-Mod}(K(A),N)$ of $\mathcal E(K(A))$-vector spaces. Hence, there is a morphism
$p_{K(A)}:
Hom_{K(A)-Mod}(K(A),N)\longrightarrow Hom_{K(A)-Mod}(K(A),M)$ of $\mathcal E(K(A))$-vector spaces such that
$p_{K(A)}\circ i_{K(A)}=1$. For any $K(A)$-module $G$, we consider the induced morphism $i_G:Hom_{K(A)-Mod}(G,M)
\longrightarrow Hom_{K(A)-Mod}(G,N)$. If $G$ is finitely presented, it can be expressed as a colimit $G\cong colim(0\longleftarrow K(A)^m
\longrightarrow K(A)^n)$ and hence $p_{K(A)}:Hom_{K(A)-Mod}(K(A),N)\longrightarrow Hom_{K(A)-Mod}(K(A),M)$ 
induces a morphism:
\begin{equation}\label{eqn2.4}
\begin{CD}
Hom_{K(A)-Mod}(G,N)\cong lim(0\rightarrow Hom_{K(A)-Mod}(K(A)^m,N)\leftarrow Hom_{K(A)-Mod}(K(A)^n,N))\\
@Vp_GVV \\ Hom_{K(A)-Mod}(G,M)\cong lim(0\rightarrow  Hom_{K(A)-Mod}(K(A)^m,M)\leftarrow Hom_{K(A)-Mod}(K(A)^n,M))\\
\end{CD}
\end{equation} such that $p_G\circ i_G=1$. Note that since $K(A)$ is projective, the morphism $p_G$ does not depend on
the choice of the presentation $G\cong colim(0\longleftarrow K(A)^m
\longrightarrow K(A)^n)$. Finally, since any $G\in A-Mod$ can be expressed as a filtered colimit of its finitely presented
submodules, we obtain a morphism $p_G:Hom_{K(A)-Mod}(G,N)\longrightarrow Hom_{K(A)-Mod}(G,M)$ such that
$p_G\circ i_G=1$ for each $G\in A-Mod$. By Yoneda Lemma, this induces  a morphism $p:N\longrightarrow M$ such that
$p\circ i=1$. 
\end{proof} 

\begin{thm}\label{PZ} Let $A\in Comm(\mathcal C)$ be a Noetherian, integral commutative monoid object.  Then, every finitely generated
$K(A)$-module is isomorphic to a direct sum $K(A)^q$ for some integer $q\geq 0$.
\end{thm} 

\begin{proof} Since monomorphisms split in $K(A)-Mod$, so do epimorphisms. Since $K(A)$ is Noetherian, any finitely generated (and hence finitely presented) 
$K(A)$-module $G$ carries an epimorphism from some $K(A)^n$. This epimorphism splits and hence we have a monomorphism 
$i:G\longrightarrow K(A)^n$. Then, $i$ induces a monomorphism $i_{K(A)}:Hom_{K(A)-Mod}(K(A),G)
\hookrightarrow Hom_{K(A)-Mod}(K(A),K(A)^n)=\mathcal E(K(A))^n$ of $\mathcal E(K(A))$-vector spaces, from which
it follows that we have an isomorphism $j_{K(A)}:Hom_{K(A)-Mod}(K(A),G)
\overset{\cong}{\longrightarrow} \mathcal E(K(A))^q=Hom_{K(A)-Mod}(K(A),K(A)^q)$ for some $q\leq n$.  Then, as in the proof
of Lemma \ref{LemmaX}, we are able to obtain isomorphisms $j_M:Hom_{K(A)-Mod}(M,G)\longrightarrow Hom_{K(A)-Mod}(M,K(A)^q)$ for
each $M\in K(A)-Mod$. By Yoneda Lemma, we now have an isomorphism $j:G\overset{\cong}{\longrightarrow} K(A)^q$. 
\end{proof}  

\begin{thm}\label{PY} Let $i:U\longrightarrow Spec(K(A))$ be a Zariski open immersion. Then, either $U=Spec(0)$ or
$i$ is an isomorphism.
\end{thm} 

\begin{proof} First we suppose that $U$ is affine, say $U=Spec(B)$ and $B\ne 0$. Since $K(A)$ has no non-trivial subobjects,
the induced map $K(A)\longrightarrow B$ is a monomorphism in $K(A)-Mod$. The monomorphism splits by Lemma \ref{LemmaX} and we
may express $B$ as a direct sum $B=K(A)\oplus T$ for some $T\in K(A)-Mod$. Then, we have:
\begin{equation}\label{Eq2.5X}
B\otimes_{K(A)}B=(K(A)\otimes_{K(A)}B)\oplus (T\otimes_{K(A)}B)=(K(A)\otimes_{K(A)}B) \oplus (T\otimes_{K(A)}K(A))\oplus 
(T\otimes_{K(A)}T)
\end{equation} Since $Spec(B)\longrightarrow Spec(K(A))$ is a Zariski immersion, $K(A)\longrightarrow B$ is an epimorphism
in $Comm(\mathcal C)$ and hence the canonical morphism $ B\otimes_{K(A)}B
\longrightarrow B\cong K(A)\otimes_{K(A)}B$ is an isomorphism. It follows that $T=T\otimes_{K(A)}K(A)=0$ and hence $B\cong K(A)$. 

\smallskip
In general, if $U$ is not affine, we choose some non-trivial Zariski open $V$ in $U$. Then, from the above reasoning, we know that $V\longrightarrow Spec(K(A))$ is an isomorphism and hence so is its  pullback $U\times_{Spec(K(A))}V\longrightarrow U$. Noticing that 
$U\times_{Spec(K(A))}V=U\times_UV=V$, we have $U\cong V\cong Spec(K(A))$ and the result follows.  

\end{proof} 

We will now show that if $X$ is a Noetherian integral scheme over $(\mathcal C,\otimes,1)$, every non-trivial Zariski
affine open $Spec(A)=U\in ZarAff(X)$ of $X$ gives us  the same field of fractions. 

\begin{thm}\label{PX} Let $X$ be a Noetherian integral scheme over $(\mathcal C,\otimes,1)$. Let $Spec(B)\longrightarrow Spec(A)$ be a morphism in $ZarAff(X)$ with $B\ne 0$. Then, $K(B)\cong B\otimes_AK(A)\cong K(A)$. 
\end{thm}

\begin{proof} From Lemma \ref{L3}, we know that any $0\ne s\in \mathcal E(A)$ is a monomorphism. Considering the filtered
colimits defining $A_s$ and $K(A)$, the canonical morphism $A\rightarrow K(A)$ is a monomorphism. Since $B$ is a flat $A$-module,
we have an induced monomorphism $B\cong B\otimes_AA\rightarrow B\otimes_AK(A)$ from which it follows that 
$B\otimes_AK(A)\ne 0$. But $Spec(B)\rightarrow Spec(A)$ being a Zariski open immersion, so is its pullback
$Spec(B\otimes_AK(A))\rightarrow Spec(K(A))$ along the morphism $Spec(K(A))\rightarrow Spec(A)$. From 
Proposition \ref{PY}, it now follows that $B\otimes_AK(A)\cong K(A)$.

Let $g:A\longrightarrow B$ be the morphism in $Comm(\mathcal C)$ underlying the morphism $Spec(B)\longrightarrow Spec(A)$
in $ZarAff(X)$. Then, since $B$ is flat and each $0\ne s\in \mathcal E(A)$ is a monomorphism, so is $s\otimes_AB\in
\mathcal E(B)$. Hence $\mathcal E(g):\mathcal E(A)\longrightarrow \mathcal E(B)$ is an injection. Then it follows that if $h:
B\longrightarrow C$ in $Comm(\mathcal C)$ takes every non-zero element in $\mathcal E(B)$ to a unit in $\mathcal E(C)$, $\mathcal E(h\circ g)$
takes every non-zero element in $\mathcal E(A)$ to a unit in $\mathcal E(C)$. From the universal property of $K(A)$, the composition
$h\circ g:A\longrightarrow C$ factors uniquely through some $h':K(A)\longrightarrow C$. The following compositions 
are now equal in $Comm(\mathcal C)$:
\begin{equation}\label{PQX}
A\overset{g}{\longrightarrow} B\overset{h}{\longrightarrow}C \qquad
A\overset{g}{\longrightarrow} B\longrightarrow B\otimes_AK(A)\cong K(A)\overset{h'}{\longrightarrow}C
\end{equation} Since $g:A\longrightarrow B$ corresponds to a Zariski open immersion, $g$ is an epimorphism
in $Comm(\mathcal C)$. It now follows from \eqref{PQX} that $h:B\longrightarrow C$ factors uniquely
through $B\otimes_AK(A)=K(A)$. From the universal property of $K(B)$, we see that
$K(B)\cong B\otimes_AK(A)\cong K(A)$.
\end{proof} 

\begin{thm} \label{PZZ} Let $X$ be a Noetherian integral scheme over $(\mathcal C,\otimes,1)$. Then, $X$ is irreducible. 
\end{thm} 

\begin{proof} Choose $U=Spec(A)\in ZarAff(X)$ with $A\ne 0$ and consider  affine opens $Spec(A_1)$, $Spec(A_2)\in ZarAff(Spec(A))
\subseteq ZarAff(X)$. As in the proof of Proposition \ref{PX}, we have a monomorphism $A\longrightarrow K(A)$
which shows that $K(A)\ne 0$. From Proposition \ref{PX}, we now note that:
\begin{equation}\label{QPX}
(A_1\otimes_AA_2)\otimes_AK(A)\cong A_1\otimes_A(A_2\otimes_AK(A))\cong A_1\otimes_AK(A)\cong K(A)\ne 0
\end{equation} from which it is clear that $A_1\otimes_AA_2\ne 0$. Hence, $Spec(A)$ is irreducible.

\smallskip
Now suppose that $X$ is not irreducible; then we can choose $Spec(B)=V\in ZarAff(X)$, $Spec(C)=W\in ZarAff(X)$
with $B\ne 0$, $C\ne 0$ such that $V\times_XW=Spec(0)$. Then, $(V\times_XU)\times_U(W\times_XU)=(V\times_XW)\times_XU
=Spec(0)$. Since $U$ is irreducible, at least one of $V\times_XU$ and $W\times_XU$ is trivial. It follows that the pullback $(V\times_XU)\coprod (W\times_XU)\longrightarrow U=Spec(A)$ 
of the canonical morphism $p:V\coprod W\longrightarrow X$ along any Zariski immersion $U=Spec(A)\longrightarrow X$
must be a Zariski immersion.  Then, $Spec(B\oplus C)=V\coprod W\in ZarAff(X)$. Hence, $\mathcal E(B\oplus C)
=\mathcal E(B)\oplus \mathcal E(C)$ must be an integral domain which is a contradiction. 

\end{proof} 

From Proposition \ref{PX} and \ref{PZZ}, it follows that for any $Spec(A)$, $Spec(B)\in ZarAff(X)$ with $A\ne 0$, $B\ne 0$, 
we have $K(A)\cong K(B)$. Hence, this common field of fractions
may be treated as the ``function field'' $K(X)$ of $X$. We also see that Propositions \ref{PZ} and \ref{PY}
further bring out the fact that $K(X)$ satisfies many properties similar to  ordinary fields, which helps justify the idea
that this common field of fractions should indeed be treated as the ``function field'' of $X$.  In \cite{AB5}, we have already constructed a field $k(X)$
for an integral scheme over $(\mathcal C,\otimes,1)$ without the Noetherian assumption.  The elements of the field $k(X)$ are equivalence classes of pairs $(U,t_U)$, with $Spec(0)\ne Spec(A)=U\in ZarAff(X)$, $t_U\in \mathcal E(A)$; for $U$, $V\in ZarAff(X)$, we say $(U,t_U)\sim 
(V,t_V)$ if there exists non-trivial $W\in ZarAff(U\times_XV)$ such that the restrictions of $t_U$ and $t_V$ to $W$
are identical. However, the object $k(X)$
obtained in \cite[$\S$ 4]{AB5} is an ordinary field, whereas in this paper we have obtained something stronger:  a commutative monoid object
$K(X)$ of $Comm(\mathcal C)$ with several field like properties as seen in Proposition \ref{PZ} and \ref{PY}. We will  show in Proposition \ref{Pop} how the field $k(X)$ constructed in \cite[$\S$ 4]{AB5} may be recovered from $K(X)$.

\smallskip
On the other hand, it is clear that an integral scheme $X$ over $(\mathcal C,\otimes,1)$  is ``reduced'', i.e., for any $Spec(A)\in ZarAff(X)$ with $A\ne 0$, $\mathcal E(A)$ must be a reduced ring. From Proposition \ref{PZZ} we see that a Noetherian integral scheme over $(\mathcal C,\otimes,1)$ is also irreducible. We can therefore say that a Noetherian integral scheme over $(\mathcal C,\otimes,1)$ is reduced and irreducible. The Noetherian hypothesis plays a key role in the results above. In essence, since our notion of integrality in Definition \ref{D1} for commutative monoid objects in $(\mathcal C,\otimes,1)$  is really ``at the 
level of global sections'', it seems that in order to obtain results analogous to those for ordinary schemes, the notion of integrality
needs to be strengthened with the additional assumption of being Noetherian. We also note that the main assumption on $(\mathcal C,\otimes,1)$ that we have used so far is that $\mathcal C$ must be locally finitely generated. We now present some examples where this conditions applies:

\medskip
\noindent {\bf Examples:} (a) For a sheaf $\mathcal A$ of rings on a topological space  $Y$  with a basis of compact open sets (say a locally Noetherian space), the  category  of sheaves of 
$\mathcal A$-modules is locally finitely generated by \cite[Theorem 3.5]{Prest}. 

\medskip
(b) If $Y=[0,1]$ and $\mathcal A_Y$ is the sheaf of continuous real valued functions on $Y$, the category
$\mathcal A_Y-Mod$ of sheaves of $\mathcal A_Y$-modules is locally finitely generated
(see \cite[Proposition 5.5]{Prest}).

\medskip
(c) If $Y$ is a topological space and $\mathcal A$ is a presheaf of commutative rings on $Y$, the category 
$\mathcal A-Premod$ of presheaves of $\mathcal A$-modules is locally finitely generated (see \cite[Corollary 2.15]{PA}). 

\medskip
We would now like to show the converse, i.e, a Noetherian scheme over $(\mathcal C,\otimes,1)$ that is reduced and irreducible is also integral. For this, we will need to make an additional assumption. First of all, we note that for any Noetherian $A\in Comm(\mathcal C)$, $A$ is a finitely generated object of $A-Mod$, i.e., the functor
$Hom_{A-Mod}(A,\_\_)$ preserves filtered colimits of monomorphisms in $A-Mod$. In order to proceed further, we will need to make the stronger assumption that any $A\in Comm(\mathcal C)$ is actually a compact object of $A-Mod$, i.e.,
the functor $Hom_{A-Mod}(A,\_\_)$ on $A-Mod$ preserves all filtered colimits in $A-Mod$ (and not just filtered colimits of 
monomorphisms). This is true, for instance, in the situation of Example (a) when the topological space is also compact, i.e., when $\mathcal C$ is the category of
$\mathcal A$-modules for a sheaf $\mathcal A$ of commutative rings on a compact topological space $Y$ with a basis of compact open sets (see \cite[Corollary 3.4]{Prest}).

\begin{thm} Let $X$ be a reduced, irreducible and Noetherian scheme over $(\mathcal C,\otimes,1)$. Suppose that
for any $A\in Comm(\mathcal C)$, $A$ is a compact object of $A-Mod$. Then, $X$ is also an integral scheme over
$(\mathcal C,\otimes,1)$. 
\end{thm} 

\begin{proof} Suppose  $X$ is not integral; then we can find some non-trivial $Spec(A)\in ZarAff(X)$ and some
$s$, $t\in \mathcal E(A)$ such that $st=0$ but $s\ne 0$ and $t\ne 0$. We will show that $A_{st}\ne 0$ which contradicts
the fact that $st=0$. Since $\mathcal E(A)$ is reduced,
neither $s$ nor $t$ is nilpotent. Hence, the ordinary localizations $\mathcal E(A)_s\ne 0$ and 
$\mathcal E(A)_t\ne 0$. Further since $A$ is a compact object of $A-Mod$, it follows from \cite[Corollary 2.8]{AB5}
that $\mathcal E(A_s)=\mathcal E(A)_s$ and $\mathcal E(A_t)=\mathcal E(A)_t$. Hence, $A_s\ne 0$
and $A_t\ne 0$. Again using the fact that $A$ is compact in $A-Mod$, it follows from \cite[Proposition 2.5]{AB5}
that $Spec(A_s)\longrightarrow Spec(A)$ and $Spec(A_t)\longrightarrow Spec(A)$ are Zariski open immersions.  Now, since $X$ is irreducible, it follows that:
\begin{equation}
Spec(A_{st})=Spec(A_s\otimes_AA_t)=Spec(A_s)\times_{Spec(A)}Spec(A_t)=Spec(A_s)\times_XSpec(A_t)\ne Spec(0)
\end{equation} 
\end{proof} 

\medskip

Our next result shows that the field $k(X)$ constructed in our previous paper \cite{AB5} may be recovered from the commutative
monoid object $K(X)$ constructed herein. 
\medskip

\begin{thm}\label{Pop} Let $X$ be a Noetherian integral scheme over $(\mathcal C,\otimes,1)$. Suppose that
for any $A\in Comm(\mathcal C)$, $A$ is a compact object of $A-Mod$. Then, $\mathcal E(K(X))\cong k(X)$. 
\end{thm}

\begin{proof} We consider some non-trivial $Spec(A)=U\in ZarAff(X)$ and a pair $(U,t_U)\in k(X)$. Then, $t_U\in \mathcal E(A)$. We know that $K(X)\cong K(A)$. From the proof of Lemma \ref{Lem2}, we know that $\mathcal E(K(A))=Q(\mathcal E(A))$, the field of fractions  of 
$\mathcal E(A)$. Hence, $t_U\in \mathcal E(A)$ corresponds to an element of $Q(\mathcal E(A))=\mathcal E(K(A))=\mathcal E(K(X))$. Conversely, any element of $\mathcal E(K(X))=Q(\mathcal E(A))$  may be expressed as a quotient $a/t$ where $a$, $t\in \mathcal E(A)$
and $t\ne 0$. But then, $a/t\in \mathcal E(A)_t=\mathcal E(A_t)$ for the Zariski affine $Spec(A_t)\in ZarAff(X)$. 
\end{proof} 

\medskip
Let $X$ and $Y$ be Noetherian integral schemes over $(\mathcal C,\otimes,1)$ and let $k(X)\longrightarrow k(Y)$
be a morphism of ordinary fields. However, such a morphism of fields does not contain enough information; in the sense
that such a morphism cannot be used to construct a corresponding (dominant, rational) map of schemes over
$(\mathcal C,\otimes,1)$ from $Y$ to $X$. As an application of our methods, we now show that this task may be
accomplished by considering the ``internal function field objects'' $K(X)$ and $K(Y)$ in $Comm(\mathcal C)$ constructed
in this paper. By a rational map from $Y$ to $X$, we will mean a morphism $\phi:V\longrightarrow X$ for some given non-trivial
$V\in ZarAff(Y)$. We will say that $\phi$ is dominant if for any non-trivial $U\in ZarAff(X)$, the pullback
$U\times_XV$ is non-trivial. 

\medskip
\begin{defn} We say that  a commutative monoid object $A\in Comm(\mathcal C)$ is of finite type if we have an isomorphism $colim_{i\in I}Hom_{Comm(\mathcal C)}
(A,A_i)\overset{\cong}{\longrightarrow}Hom_{Comm(\mathcal C)}(A,colim_{i\in I}A_i)$   for any filtered system 
$\{A_i\}_{i\in I}$ in $Comm(\mathcal C)$. 

\medskip
A scheme $X$ will be said to be of finite type over $(\mathcal C,\otimes,1)$ if $A\in Comm(\mathcal C)$ is of finite type for each $Spec(A)\in ZarAff(X)$;  
\end{defn}

\medskip

\begin{Thm}Suppose that
for any $A\in Comm(\mathcal C)$, $A$ is a compact object of $A-Mod$.  Let $X$, $Y$ be  Noetherian integral schemes of finite type over $(\mathcal C,\otimes,1)$. Then, there exists a bijection between morphisms $K(X)\longrightarrow K(Y)$ in $Comm(\mathcal C)$ and dominant rational maps from $Y$ to $X$. 
\end{Thm}

\begin{proof} We consider a morphism $K(X)\overset{g}{\longrightarrow} K(Y)$ and choose some $Spec(A)=U\in ZarAff(X)$, $Spec(B)=V
\in ZarAff(Y)$. We consider the induced morphism $A\longrightarrow K(A)\cong K(X)\overset{g}{\longrightarrow}K(Y)\cong K(B)
=colim_{t\in \mathcal E(B)\backslash \{0\}}B_t$. Since $A$ is of finite type, this morphism factors through $B_t$ 
for some $0\ne t\in \mathcal E(B)$. Then since $Y$ is irreducible,   $Spec(B_t)\in ZarAff(Y)$ is dense in $Y$ and we obtain a rational
map $\phi:V_t:=Spec(B_t)\longrightarrow Spec(A)\longrightarrow X$ from $Y$ to $X$. If $\phi$  is not dominant, there exists 
non-trivial $U'\in ZarAff(X)$ such that $V_t\times_XU'=Spec(0)$. Then, for any  $W=Spec(C)\in ZarAff(U\times_XU')$, we must have
$B_t\otimes_AC=0$.  Since $K(A)$ has no non-zero proper subobjects, $K(A)\longrightarrow K(B)$ is a monomorphism. Then, since $K(A)$ is a flat $A$-module and $K(A)\otimes_AK(A)\cong K(A)$, we obtain a contradiction by considering the monomorphism:
\begin{equation*}
0\ne K(A)\cong K(A)\otimes_AK(A)\hookrightarrow K(B)\otimes_AK(A)=K(B_t)\otimes_AK(C)=K(B_t)\otimes_{B_t}(B_t\otimes_AC)\otimes_{C}K(C)=0
\end{equation*} Conversely, given a dominant rational map $\phi:V\longrightarrow X$ for some $V\in ZarAff(Y)$, the pullback
$U\times_XV$ is non-trivial for any $Spec(0)\ne Spec(A)=U\in ZarAff(X)$. Then, by choosing non-trivial $Spec(B')=V'\in ZarAff(U\times_XV)$, we obtain an induced morphism $V'=Spec(B')\longrightarrow Spec(A)=U$. The latter corresponds to a morphism $A\longrightarrow B'$ in $Comm(\mathcal C)$, which we denote by $\varphi:A\longrightarrow B'$. Now suppose that there exists
$0\ne s\in \mathcal E(A)$ such that $\mathcal E(\varphi)(s)=0\in \mathcal E(B')$. We now set $U':=Spec(A_s)$. Since 
$\phi:V\longrightarrow X$ is dominant, we know that $U'\times_XV$ is non-trivial and we choose some non-trivial
$V''\in ZarAff(U'\times_XV)$. Then, since $\mathcal E(\varphi)(s)=0\in \mathcal E(B')$, the intersection 
$V''\times_XV'$ must be trivial, which contradicts the fact that $Y$ is irreducible. Hence, it follows that
$\mathcal E(\varphi)(s)\ne 0$ for each $0\ne s\in \mathcal E(A)$. Accordingly, the morphism 
$\varphi:A\longrightarrow B'$ in $Comm(\mathcal C)$ now induces a morphism
$K(X)\cong K(A)\longrightarrow K(B')\cong K(Y)$.
\end{proof}

\end{document}